\documentclass[letterpaper,10pt]{article}
\usepackage[textwidth=11.5cm, textheight=22cm, top=2.0cm, centering]{geometry}
\usepackage{setspace}
\usepackage[T1]{fontenc}
\usepackage{lmodern}
\usepackage[english]{babel}
\usepackage{microtype} 
\usepackage{euscript}
\DeclareMathAlphabet{\mathcal}{U}{eus}{m}{n} 
\usepackage[scaled=0.95]{inconsolata} 
\makeatletter
\renewcommand{\texttt}[1]{{\ttfamily #1}}

\renewcommand{\mathtt}[1]{\text{\texttt{#1}}} 
\makeatother
\usepackage{csquotes}
\usepackage{etoolbox}
\usepackage{bbding}
\usepackage{changepage}
\usepackage{booktabs}
\usepackage{tabularx}

\usepackage{graphicx}
\usepackage[x11names, dvipsnames]{xcolor}
\definecolor{Linkz}{RGB}{30, 110, 170}
\definecolor{Darkenta}{RGB}{185, 35, 90}
\definecolor{Magentz}{RGB}{255, 35, 170}
\definecolor{Lightenta}{RGB}{254, 232, 255}
\definecolor{Reference}{RGB}{35, 180, 90}
\definecolor{Periwinkle}{RGB}{102, 51, 255}
\definecolor{yello}{RGB}{255, 245, 230}
\definecolor{Greeno}{RGB}{0, 140, 100}
\definecolor{Leeno}{RGB}{239, 255, 232}
\definecolor{Nicegreen}{RGB}{100, 200, 130}
\usepackage{amsmath, amsthm, amscd, amssymb, stmaryrd}
\usepackage{stackengine}
\usepackage{bussproofs}
\usepackage[square]{natbib}

\usepackage[
colorlinks=true,
linkcolor=Darkenta,
citecolor=Greeno,
urlcolor=Darkenta,
]{hyperref}
\usepackage{titlesec} 
\usepackage{abstract} 
\usepackage{enumitem} 
\usepackage{ccicons} 
\usepackage{float}
\usepackage{placeins} 
\usepackage{tikz} 
\usepackage{tikz-cd}
\usetikzlibrary{positioning,fit,patterns}
\usetikzlibrary{arrows.meta} 
\usepackage{adjustbox} 
\usepackage{pgfmath} 
\usepackage{ifthen} 

\newtheoremstyle{upright}
{6pt plus 2pt minus 2pt} 
{6pt plus 2pt minus 2pt} 
{\normalfont} 
{} 
{\bfseries} 
{.} 
{.5em} 
{} 
\theoremstyle{upright}
\theoremstyle{upright}
\newtheorem{theorem}{Theorem}[subsection]
\newtheorem{remark}[theorem]{Remark}

\newtheorem{definition}[theorem]{Definition}

\newtheorem{proposition}[theorem]{Proposition}
\newtheorem{lemma}[theorem]{Lemma}
\newtheorem{conjecture}[theorem]{Conjecture}

\newtheorem{corollary}[theorem]{Corollary}
\newtheorem{example}[theorem]{Example}

\newtheorem{criterion}[theorem]{Criterion}

\makeatletter
\renewenvironment{proof}[1][Proof]{%
	\par\pushQED{\qed}%
	\normalfont
	\topsep6\p@\@plus6\p@\relax
	\trivlist
	\item[\hskip\labelsep\slshape #1\@addpunct{.}]%
}{%
	\popQED\endtrivlist\@endpefalse
}

\makeatother

\usepackage{algpseudocode}
\usepackage[most]{tcolorbox}
\usepackage{listings}
\newtcolorbox{breakbox}[2][]{%
	breakable,
	={#2},
	fonttitle=\bfseries,
	colback=white,
	colframe=black!20,
	coltitle=black,
	colbacktitle=white,
	boxrule=0.5pt,
	arc=0pt,
	boxsep=7pt,
	left=3pt,
	right=2pt,
	top=2pt,
	bottom=4pt,
	fontupper=\small\sffamily, 
	#1
}

\AtBeginEnvironment{quotation}{\sffamily}
\renewenvironment{quotation}
{\small\vspace{0.5em}\begin{adjustwidth}{4em}{4em}%

		\setlength{\parindent}{0pt}%
		\setlength{\parskip}{\medskipamount}%
	}
	{\end{adjustwidth}\vspace{1em}}

\AtBeginEnvironment{prose}{\sffamily}

\newcommand{\customsectionstyle}[2]{%
	\titleformat{\section}[block]
	{\normalfont\fontsize{#1}{1.2\dimexpr#1\relax}\selectfont\centering}
	{\thesection}{1em}%
	{%
		\ifthenelse{\equal{#2}{true}}{\MakeUppercase}{\relax}%
	}%
}
\newcommand{\customsectionspacing}[3]{%
	\titlespacing*{\section}{#1}{#2}{#3}%
}
\newcommand{\customsubsectionstyle}[2]{%
	\titleformat{\subsection}[block]
	{\normalfont\fontsize{#1}{1.2\dimexpr#1\relax}\selectfont\centering}
	{\thesubsection}{1em}%
	{%
		\ifthenelse{\equal{#2}{true}}{\MakeUppercase}{\relax}%
	}%
}
\newcommand{\customsubsectionspacing}[3]{%
	\titlespacing*{\subsection}{#1}{#2}{#3}%
}
\customsectionstyle{14pt}{true}
\customsectionspacing{0pt}{3.5ex plus 1ex minus 0.2ex}{2.5ex}
\customsubsectionstyle{11pt}{true}
\customsubsectionspacing{0pt}{4ex plus 1ex minus 0.2ex}{2ex}
\usepackage{caption}

\makeatletter
\newcommand{\shorttitle}[1]{\def\@shorttitle{#1}}
\newcommand{\email}[1]{\def\@email{#1}}
\newcommand{\metadata}[1]{\def\@metadata{#1}}
\renewcommand{\maketitle}{%
	\begin{center}
		\vfill
		{\fontsize{18pt}{19pt}\selectfont \@title \par}
		\vspace{1em}
		{\normalsize \@author \par}
		\vspace{0.1em}
		{\normalsize \@date \par}
	\end{center}
}
\makeatother

\newcommand{\HA}{\mathsf{HA}}
\newcommand{\Prov}{\mathsf{Prov}}
\newcommand{\Con}{\mathsf{Con}}
\newcommand{\Div}{\mathsf{Div}}
\newcommand{\code}[1]{\ulcorner #1\urcorner}
\newcommand{\expr}[1]{\llbracket #1\rrbracket}
\begin{document}
\title{\uppercase{Adversarial Problems and the Limits of Uniform Resolution}}
\author{Milan Rosko}
\date{October 2026}

\maketitle

\begin{center}\scriptsize
	ORCID: \href{https://orcid.org/0009-0003-1363-7158}{\textsf{0009-0003-1363-7158}}
\end{center}

\begin{abstract}
	\vspace{-1.5ex}\footnotesize{
		A general method of resolution can take part in forming problems once their language represents the method itself. A fixed formation rule gives each sound resolver a presentation it cannot answer. Admitting the diagnosis yields a sound successor with an instance of its own. For standard arithmetical theories $\mathrm{T}$, the instance is $\HA$-provably equivalent to $\Con(\mathrm{T})$. A theory that decides its own instance asserts a false sentence. The $m$-th iterated independence question is equivalent to the instance $m$ stages later. No effective admission policy with consistent union closes the regress. A noncomputable answer function resolves all program indices, while a false assertion of resolution supplies no warrant. At feasible scale, bounded self-consistency is cheap. A clocked Gödel sentence gives an unconditional gap relative to its short binary description, but no superpolynomial lower bound in the window size. The arithmetic identity motivates studying bounded consistency of the successor theory, as in a conjecture of Pudlák. We state its finitely axiomatized case and prove that it implies $\mathcal{NP}\ne\mathrm{co}\mathcal{NP}$. The conjecture motivates a feasible analogue, but does not establish all clauses of the unbounded barrier.
	}
\end{abstract}

\section{Introduction}
\label{sec:introduction}

\begin{figure}[ht]
    \centering
    \includegraphics[width=0.66\textwidth]{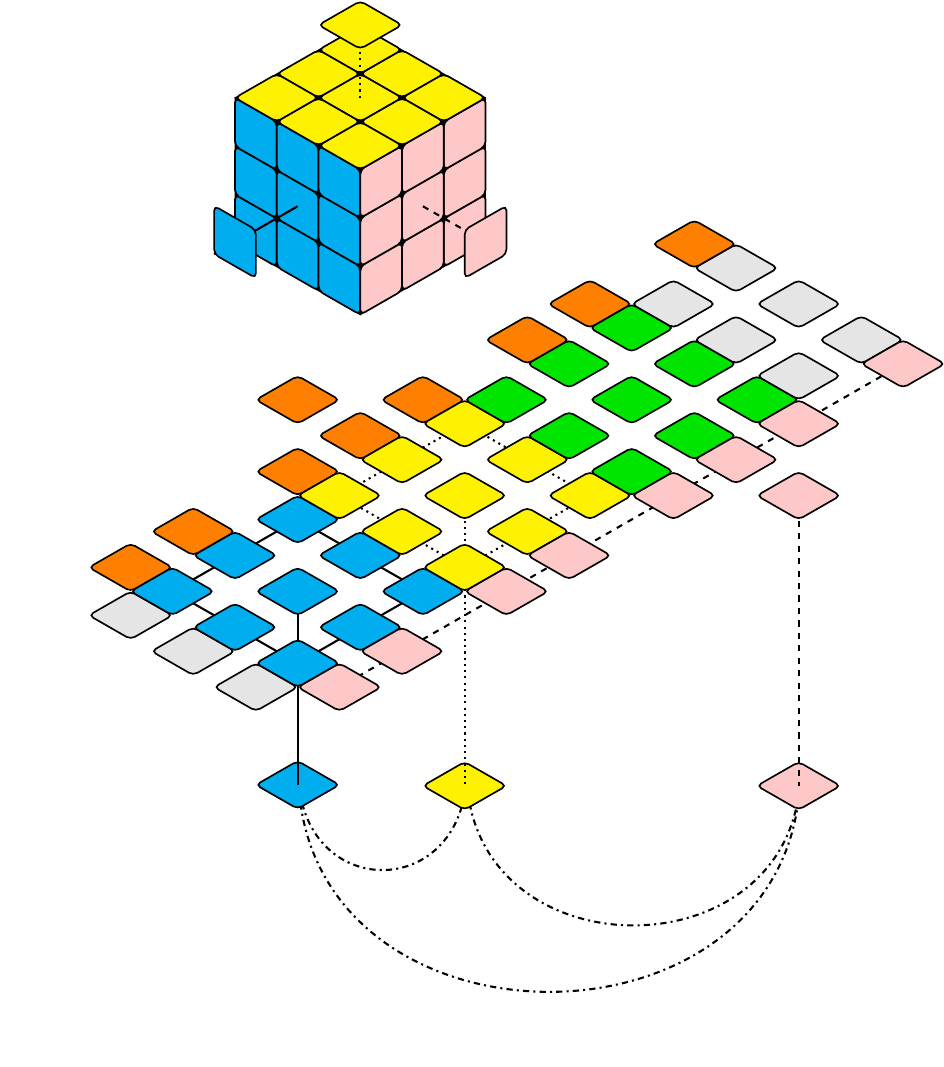}
    \caption{From a puzzle to morphisms. The \textsc{Rubik's Cube} and its separated cells suggest a representation of states, transformations, and composition. A general solving method operates on this representation, which can also express demands concerning the method itself, as a machine.}
    \label{fig:rubik}
\end{figure}

A sequence of operations may solve a particular task. A general method requires more: a representation of states, transformations, and conditions of success. Once these are expressed in a common language, operations can be composed, their behavior studied, and claims about their adequacy formulated. Crucially, a method represented in this language may itself become the object of a demand formed within it (Figure~\ref{fig:rubik}).

This raises the question:
\begin{quotation}
    What happens when a proposed resolution becomes available within the practice whose demands it claims to settle? Can a certificate that a problem has no solution be obstructed by its own availability?
\end{quotation}

We study a precise form of this phenomenon, termed here as a \emph{reflexive barrier}: under specified conditions, the warrant for an obstruction is itself an instance of that obstruction. Definition~\ref{def:reflexive} distinguishes failure of an operation to answer, failure of a practice to certify that failure, the effect of admitting an external certificate, and the status of the barrier statement itself.

The technical ingredients are classical: the \textsc{Recursion Theorem}, productive sets, \textsc{Gödel's Second Incompleteness Theorem}~\citep{goedel31}, \textsc{Kleene Realizability}~\citep{kleene45}, and results on consistency and proof systems~\citep{krajicekpudlak89,pudlak17}. The contribution is their organization under a formation-relative semantics, together with the arithmetic identity that expresses independence as the next trap. Productiveness supplies obstruction (B1), inner collapse extends to unsound standard theories (B2), and the independence tower supplies recurrence of the warrant (B3$'$).

The effective construction begins with a resolver that, given a program index, returns either a terminating computation with binary output or a registered divergence certificate. A uniform formation rule, obtained from \textsc{Kleene's Recursion Theorem}~\citep{kleene52}, associates with each sound resolver a program that reacts adversely to the resolver's answer on its own index. The resulting program diverges, but the resolver cannot answer it. We call such instances \emph{adversarial}. Registering their divergence yields a sound successor, to which the same formation rule applies. This is the productiveness of the divergent indices~\citep{rogers67}, interpreted as formation, diagnosis, and admission. The construction requires program codes, specialization, and execution of the represented resolver. Strong normalization excludes this unrestricted execution internally, but does not exclude representing codes and finite computation traces.

The arithmetical construction takes the divergence registry to be the divergence theorems of a theory $\mathrm{T}$ extending \textsc{Heyting Arithmetic}. The adversarial instance is then $\HA$-provably equivalent to $\Con(\mathrm{T})$. More strongly, its unprovability in $\mathrm{T}$ is itself $\HA$-provably equivalent to $\Con(\mathrm{T})$: the assertion that the trap is unprovable is the trap. A theory that proves its trap is inconsistent. A consistent theory can refute its trap, but then asserts a false sentence in the reference class. Admitting the trap passes from $\mathrm{T}$ to $\mathrm{T}+\Con(\mathrm{T})$. The assertion that the original trap is independent of $\mathrm{T}$ is, again provably in $\HA$, the trap of this successor. Iteration yields a sequence in which the $m$-th iterated independence question is equivalent to the trap $m$ stages later (Theorem~\ref{thm:tower}).

Every trap of a consistent theory is true but unprovable in that theory. Proof search can recognize a sentence's equivalence to the trap, but its truth requires an external consistency warrant. No effective admission policy with consistent union removes this obstruction. The construction also admits a Diophantine presentation: the trap becomes the assertion that a polynomial has no root.

At feasible scale, two features of the unbounded argument fail. First, diagonalization has a size cost: a self-describing instance must encode the operation it invokes and can exceed the resource window it purports to address. Second, bounded consistency is comparatively easy: sufficiently strong theories and proof systems admit short proofs of their own bounded consistency~\citep{pudlak86,buss91,pudlak20}. The clocked \textsc{Gödel Sentence} underlying classical speed-up gives an unconditional bounded obstruction, but its gap is measured against the length of the description of the resource bound, not against the bound itself.

These limitations motivate studying the bounded consistency of a successor theory. The arithmetic independence identity leads to a conjecture of \citet{pudlak20} concerning proofs in $\mathrm{T}$ of
$\Con_{\mathrm{T}+\Con(\mathrm{T})}(\overline n)$.
We formulate its finitely axiomatized case and show that it implies $\mathcal{NP}\ne\mathrm{co}\mathcal{NP}$. A corresponding propositional construction, in which admission adds a proof system's reflection principle, fails because strong systems can prove their own reflection. We neither establish Pudlák's conjecture nor claim that it satisfies every clause of Definition~\ref{def:reflexive}. Nor do these results establish the independence, undecidability, or resolution of $\mathcal P$ versus $\mathcal{NP}$; Criterion~\ref{crit:hierarchy} specifies the limitation.

The development is \emph{formation-relative}: quantification ranges over formation derivations within a specified practice. The closure result covers effective admission policies without choosing ordinal notations. We retain intuitionistic arguments where sufficient and identify classical steps explicitly.

Section~\ref{sec:practice} defines practices, evidence, and warrants. Section~\ref{sec:reflexive} introduces reflexive barriers. Sections~\ref{sec:effective} and~\ref{sec:arithmetic} establish the effective and arithmetical constructions, including the tower and inner collapse. Section~\ref{sec:closure} proves nonclosure under effective admission policies. Section~\ref{sec:feasible} examines the feasible regime.

We hedge once and for all: we do not prove the conjecture, we do not claim that it is an instance of every clause of Definition~\ref{def:reflexive}, and we do not claim that $\mathcal P$ versus $\mathcal{NP}$ is independent, undecidable, or resolved; Criterion~\ref{crit:hierarchy} explains why no argument of the kind given here could support such a claim.

\section{Practices and Warrants}
\label{sec:practice}

\subsection{The answer contract}

\begin{definition}[Presentations and evidence]
	\label{def:practice}
	A practice $\Gamma$ specifies presentations $x$, formation derivations $d:\mathsf{Adm}_{\Gamma}(x)$, and accepted evidence types $E_i(x)$, indexed by a finite set $I$. Distinct answers are incompatible: for $i\ne j$ there is a refuter
	\begin{equation}
		r_{i,j,x}:E_i(x)\to E_j(x)\to\bot.
		\label{eq:incompatibility}
	\end{equation}
	Here $\bot$ denotes contradiction. Sums and dependent sums record choices and their witnesses; products record uniform constructions. The notation is read in a constructive metatheory of programs and finite computations.

	A formation derivation records why a presentation enters the stated practice. It can contain an answer witness, but formation need not supply one. Quantification over problems is always quantification over formation derivations of a specified practice; no completed domain of problems is assumed.
\end{definition}

\begin{definition}[Result and warrant]
	\label{def:warrant}
	Let $\mathsf{Eval}(C,x,(i,w))$ record a terminating evaluation of an operation $C$ on the presentation $x$ with returned answer $(i,w)$. A decoder $\mathsf{ev}_{\Gamma,i,x}(w)$ interprets an accepted certificate $w$ as evidence in $E_i(x)$. Write $\mathsf{Cert}_{\Gamma,i}(x)$ for the accepted certificates. For $d:\mathsf{Adm}_{\Gamma}(x)$, set
	\begin{equation}
		\begin{aligned}
			\mathsf{Result}_{\Gamma}(C,d)
			&:=\sum_{i:I}\sum_{w:\mathsf{Cert}_{\Gamma,i}(x)}
			\mathsf{Eval}(C,x,(i,w)),\\
			\mathsf{W}_{\Gamma}(C)
			&:=\prod_{x}\prod_{d:\mathsf{Adm}_{\Gamma}(x)}
			\mathsf{Result}_{\Gamma}(C,d).
		\end{aligned}
		\label{eq:warrant}
	\end{equation}
	A warrant supplies a result for each admitted presentation. A result depends on its derivation only through the presentation $x$; the derivation records that $x$ lies in the scope of the warrant.
\end{definition}

\begin{remark}
	\label{rem:closure}
	The demand considered throughout has a fixed form:
	\begin{quotation}
		Supply accepted answer evidence for each admitted presentation, including presentations formed by the stated rule from the proposed resolving operation.
	\end{quotation}
	For an operation $C$, write $\Gamma\langle C\rangle$ for the practice extended by that formation rule; it retains every formation of $\Gamma$ and changes neither the accepted certificates nor their decoders. The corresponding closure claim is
	\begin{equation}
		\mathsf{Cl}_{\Gamma}(C)
		:=\mathsf{W}_{\Gamma\langle C\rangle}(C).
		\label{eq:closure}
	\end{equation}
	The formation rule and answer meanings are fixed before a candidate is selected.
\end{remark}

\subsection{Admission and the local argument}

\begin{definition}[Admission]
	\label{def:package}
	A package $a$ consists of asserted content, evidence for it, finite data made available to operations, and a stated scope of use. Admission produces a practice $\Gamma[a]$ and an operation $C[a]$ with specified certificate and formation rules. Existing presentations and evidence retain their interpretations.
\end{definition}

\begin{lemma}[Local obstruction]
	\label{lem:local}
	Fix $d:\mathsf{Adm}_{\Gamma}(x)$. Suppose that for each accepted result branch $i$, its evidence and evaluation produce evidence for an incompatible branch $j(i)\ne i$. Then
	\begin{equation}
		\neg\mathsf{Result}_{\Gamma}(C,d),
		\quad
		\neg\mathsf{W}_{\Gamma}(C).
		\label{eq:local}
	\end{equation}
\end{lemma}

\begin{proof}
	Decode the certificate in a proposed result to obtain $u:E_i(x)$. The assumed branch construction produces $v:E_{j(i)}(x)$. Apply $r_{i,j(i),x}(u,v)$. This refutes the result in every branch. A proposed warrant supplies precisely such a result by evaluation at $x,d$.
\end{proof}

\section{Reflexive Barriers}
\label{sec:reflexive}

\subsection{The definition}

\begin{definition}[Resolution setting]
	\label{def:setting}
	A resolution setting consists of the following data.
	\begin{enumerate}[label=\textnormal{(\roman*)}]
		\item A class $\mathcal K$ of pairs $(\Gamma,C)$ of a practice and an operation $C$ with a code $\code{C}$.
		\item A formation rule $F$ assigning to each $(\Gamma,C)\in\mathcal K$ a derivation
		\begin{equation}
			d_F(C):\mathsf{Adm}_{\Gamma\langle C\rangle}\bigl(F(C)\bigr),
			\label{eq:formation-rule}
		\end{equation}
		computed from $\code{C}$. We call $F(C)$ the adversarial presentation of $C$.
		\item An admission rule sending $(\Gamma,C)$ and a package $a$ to $(\Gamma[a],C[a])$.
		\item A base theory $\mathsf B$ with a realizability interpretation $\Vdash$ of its sentences, and a reference class $\mathcal R$ of sentences on which realizability coincides with truth. A sentence $\varphi$ has a realizer if $e\Vdash\varphi$ for some $e$.
	\end{enumerate}
	The barrier statement at $(\Gamma,C)$ is
	\begin{equation}
		\mathsf{Bar}(\Gamma,C)
		:=\neg\mathsf{Result}_{\Gamma\langle C\rangle}\bigl(C,d_F(C)\bigr).
		\label{eq:barrier}
	\end{equation}
	The practice $\Gamma$ is assertional if it has an assertion relation $\Gamma\vdash\varphi$ on sentences of $\mathsf B$, expressible in $\mathsf B$ by a formula $\mathsf{Ass}_\Gamma$, and if $\mathsf{Bar}(\Gamma,C)$ is expressed by a sentence $\expr{\mathsf{Bar}(\Gamma,C)}$ of $\mathsf B$; the notation $\code{\cdot}$ is reserved for codes. For assertional practices write
	\begin{equation}
		\mathsf{Indep}_{\Gamma}(\varphi)
		:=\neg\mathsf{Ass}_\Gamma(\code{\varphi})
		\wedge
		\neg\mathsf{Ass}_\Gamma(\code{\neg\varphi}).
		\label{eq:indep}
	\end{equation}
\end{definition}

\begin{definition}[Recurrent and reflexive barriers]
	\label{def:reflexive}
	Let a resolution setting be given, and suppose that for every $(\Gamma,C)\in\mathcal K$ a package $a_C$, the diagnosis at $C$, is specified. Consider the clauses:
	\begin{enumerate}[label=\textnormal{(B\arabic*)}]
		\item \emph{Obstruction.} $\mathsf{Bar}(\Gamma,C)$ holds.
		\item \emph{Inner collapse.} If $\Gamma$ asserts $\expr{\mathsf{Bar}(\Gamma,C)}$ or its negation, then $\Gamma$ asserts a sentence of $\mathcal R$ without a realizer.
		\item \emph{Outer recurrence.} $(\Gamma[a_C],C[a_C])\in\mathcal K$, the practice $\Gamma[a_C]$ admits $F(C)$, and $\mathsf{W}_{\Gamma[a_C]}(C[a_C])$ holds.
		\item[\textnormal{(B3$'$)}] \emph{Recurrence of the warrant.} $\mathsf B$ proves
		\begin{equation}
			\mathsf{Indep}_{\Gamma}\bigl(\expr{\mathsf{Bar}(\Gamma,C)}\bigr)
			\leftrightarrow
			\expr{\mathsf{Bar}(\Gamma[a_C],C[a_C])}.
			\label{eq:B3prime}
		\end{equation}
		\item \emph{Preservation.} Every presentation admitted in $\Gamma\langle C\rangle$ is admitted in $\Gamma[a_C]$, and every accepted certificate is carried to an accepted certificate of $\Gamma[a_C]$ with the same evidence.
	\end{enumerate}
	The formation rule $F$ is a recurrent barrier on $\mathcal K$ if (B1), (B3), and (B4) hold for every $(\Gamma,C)\in\mathcal K$. It is a reflexive barrier if, in addition, every practice in $\mathcal K$ is assertional and (B2) and (B3$'$) hold.
\end{definition}

\begin{remark}
	The clauses separate four things that are easily conflated. (B1) concerns the operation: it has no result on the formed presentation. (B2) concerns the practice: it can decide its own barrier only by asserting something without a realizer. In the arithmetical instance below, the two cases of (B2) are inconsistency and the false assertion of a contradiction's existence; realizability treats both uniformly, since neither has a realizer. The reference class is needed because a classical practice asserts instances of excluded middle that have no realizer; restricting (B2) to $\mathcal R$, where realizability and truth agree, makes the clause meaningful for classical and intuitionistic practices alike. (B3) concerns admission: supplying the diagnosis from outside answers the instance. (B3$'$) concerns the warrant: the statement that the barrier is undecided at one stage is, over the base, the barrier at the next. Together, (B2) and (B3$'$) state that the warrant of the obstruction is an instance of it.
\end{remark}

\subsection{Elementary consequences}

\begin{proposition}[No closure]
	\label{prop:no-closure}
	If (B1) holds at $(\Gamma,C)$, then $\neg\mathsf{Cl}_{\Gamma}(C)$.
\end{proposition}

\begin{proof}
	A closure warrant supplies a result on the derivation $d_F(C)$, which (B1) refutes.
\end{proof}

\begin{proposition}[Recurrence]
	\label{prop:recurrence}
	Let $F$ be a recurrent barrier on $\mathcal K$ and $(\Gamma_0,C_0)\in\mathcal K$. Define
	\begin{equation}
		(\Gamma_{n+1},C_{n+1})
		:=\bigl(\Gamma_n[a_{C_n}],C_n[a_{C_n}]\bigr).
		\label{eq:recurrence}
	\end{equation}
	Then every stage lies in $\mathcal K$, every stage answers the presentation formed against its predecessor, every stage has a presentation it cannot answer, and the presentations $F(C_n)$ are pairwise distinct.
\end{proposition}

\begin{proof}
	Membership follows from (B3) by induction, and (B1) supplies the unanswered presentation at each stage. Let $j<n$. By (B3), $\Gamma_{j+1}$ admits $F(C_j)$. Since $\Gamma_i\langle C_i\rangle$ retains the formations of $\Gamma_i$, (B4) carries this admission to every later stage, and the warrant of (B3) at stage $n$ supplies a result of $C_n$ on $F(C_j)$. By (B1), $C_n$ has no result on $F(C_n)$. Since $\Gamma_n\langle C_n\rangle$ changes neither certificates nor decoders, and results depend on derivations only through presentations, $F(C_j)\ne F(C_n)$.
\end{proof}

\begin{proposition}[Self-warrant]
	\label{prop:self-warrant}
	Let $F$ be a reflexive barrier and $(\Gamma,C)\in\mathcal K$. If every sentence of $\mathcal R$ asserted by $\Gamma$ is true, then $\Gamma$ asserts neither $\expr{\mathsf{Bar}(\Gamma,C)}$ nor its negation.
\end{proposition}

\begin{proof}
	On $\mathcal R$, truth and realizability coincide; apply (B2).
\end{proof}

\begin{remark}
	Proposition~\ref{prop:recurrence} is the abstract form of the recurrence: answering one presentation is the step that makes the next one formable. Proposition~\ref{prop:self-warrant} with (B3$'$) is the abstract form of the thesis: a practice that is truthful on $\mathcal R$ leaves its barrier undecided, and the statement of that undecidedness is the barrier of the next stage. Section~\ref{sec:effective} gives a recurrent barrier on programs; Section~\ref{sec:arithmetic} gives a reflexive barrier on arithmetical theories.
\end{remark}

\section{The Effective Instance}
\label{sec:effective}

\subsection{Formation, completion and diagnostic evidence}

\begin{example}
	A completion problem includes the permitted pieces, matching conditions, boundary conditions, and operations by which further tasks may be presented. These determine what counts as a solution. Figure~\ref{fig:penrose} places an altered patch with openings, a patch awaiting completion, and a completed pattern side by side. It makes the dependence on formation visible: the geometry offered to a solver belongs to its obligation.

	\begin{center}
		\includegraphics[width=\textwidth]{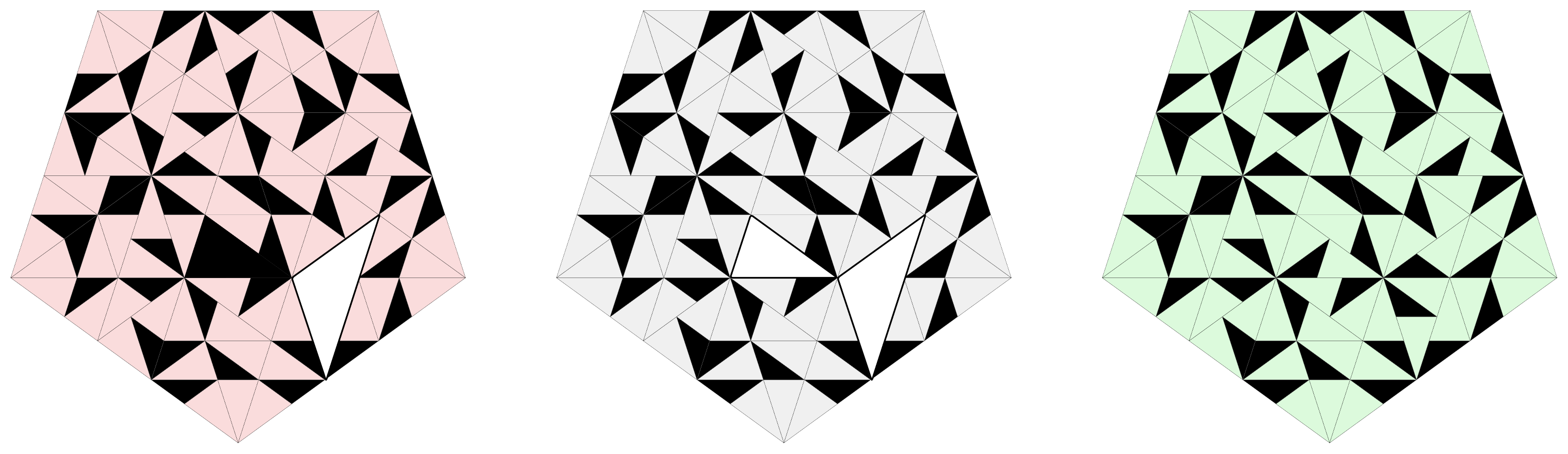}
		\captionof{figure}{Formation and completion. Schematic patches based on \textsc{Penrose Tilings}: (i) an altered patch with openings, (ii) a patch awaiting completion, (iii) and an exhibited completed pattern. Their interpretation depends on the supplied formation and matching rules. In the construction below, a proposed resolving operation participates in the formation of the task it is asked to resolve.}
		\label{fig:penrose}
	\end{center}

	Once an operation is represented in the problem language, the rules can form a task whose behavior depends on the answer that operation supplies. The following construction makes this dependence explicit.
\end{example}

\begin{definition}[Three accepted answers]
	\label{def:answers}
	Fix the conventional numbering $(\varphi_e)_{e\in\mathbb N}$ of deterministic register-machine programs, with primitive-recursive coding of programs and finite computation traces. We use a primitive-recursive specialization function $s(a,z)$: it compiles the first argument $z$ into the two-argument program with code $a$. The trace predicate and compiler are chosen so that $\HA$ proves
	\[
		\varphi_{s(a,z)}(x)\simeq\varphi_a(z,x),
	\]
	where equality of partial values means equivalence of termination with each output. This internal correctness condition, rather than external acceptability alone, is part of the coding convention. Presentations are program indices, evaluated at input $0$. If a computation returns $y$, its binary outcome is $y\bmod 2$. Let $\mathsf T(e,0,t,y)$ be the primitive recursive relation stating that $t$ encodes a terminating computation of $\varphi_e(0)$ with output $y$. Define
	\begin{equation}
		\begin{aligned}
			E_i(e)&:=\sum_{t,y}
			\bigl(\mathsf T(e,0,t,y)\mathbin{\times}(y\bmod2=i)\bigr),
			\quad i\in\{0,1\},\\
			E_{\infty}(e)&:=\bigl(E_0(e)+E_1(e)\bigr)\to\bot.
		\end{aligned}
		\label{eq:three-answers}
	\end{equation}
	Determinism makes $E_0(e)$ and $E_1(e)$ incompatible, and $E_{\infty}(e)$ is incompatible with each by definition. We write $\varphi_e(0)\uparrow$ for the conclusion of $E_{\infty}(e)$.
\end{definition}

\begin{definition}[Registry and resolver]
	\label{def:registry}
	A registry is an index $k$ with a decidable enumeration relation $\mathsf{Enum}(k,s,e)$, read as \enquote{$e$ is registered at stage $s$}. Write
	\begin{equation}
		W_k:=\{e:\exists s\,\mathsf{Enum}(k,s,e)\}.
		\label{eq:registry-set}
	\end{equation}
	The registry is sound when it carries a construction
	\begin{equation}
		\delta_k:\prod_{e}\prod_{s}
		\bigl(\mathsf{Enum}(k,s,e)\to E_{\infty}(e)\bigr).
		\label{eq:registry-soundness}
	\end{equation}
	The diagnostic certificate $(\infty,s)$ on input $e$ is accepted when $\mathsf{Enum}(k,s,e)$ holds and denotes $\delta_k(e,s)$.

	The partial recursive resolver $C_k$ inspects stages $n=0,1,2,\ldots$ on input $e$. At stage $n$ it returns $(\infty,n)$ if $\mathsf{Enum}(k,n,e)$ holds; otherwise it returns $(y\bmod2,n)$ if $\mathsf T(e,0,n,y)$ holds for some $y$. The practice $\Gamma_k$ admits $e$ with either a terminating computation witness or a registration:
	\begin{equation}
		\mathsf{Adm}_{\Gamma_k}(e)
		:=E_0(e)+E_1(e)
		+\sum_{s}\mathsf{Enum}(k,s,e).
		\label{eq:base-formation}
	\end{equation}
	The resolver receives only the index $e$, not the formation witness.
\end{definition}

\begin{proposition}[The starting warrant]
	\label{prop:starting}
	For every sound registry $k$, every terminating output of $C_k$ is a valid answer on its input, and
	\begin{equation}
		w_k:\mathsf{W}_{\Gamma_k}(C_k).
		\label{eq:starting-warrant}
	\end{equation}
\end{proposition}

\begin{proof}
	A diagnostic return $(\infty,n)$ is issued only when $\mathsf{Enum}(k,n,e)$ holds, and $\delta_k(e,n)$ supplies its evidence. A computational return carries its own witness. For a formation witnessed by a registration at stage $s$, or by a computation encoded by $t$, the decidable conditions inspected by $C_k$ hold by stage $s$, respectively $t$, so the search terminates with a sound answer.
\end{proof}

\subsection{The fixed formation rule}

\begin{definition}[Adversarial program]
	\label{def:diagonal}
	For a program code $c$, the partial program $B(c,z)$ runs $\varphi_c(z)$, reads a returned answer tag, and responds by
	\begin{equation}
		\begin{array}{c|c|c}
			\text{returned answer}&\text{action of }B(c,z)&\text{outcome evidence}\\
			\hline
			(0,w)&\text{return }1&E_1\text{ for its own run}\\
			(1,w)&\text{return }0&E_0\text{ for its own run}\\
			(\infty,w)&\text{return }0&E_0\text{ for its own run}
		\end{array}
		\label{eq:branch-table}
	\end{equation}
	It diverges if the called operation diverges or returns an ill-formed answer. Using the specialization function fixed in Definition~\ref{def:answers}, the parameterized \textsc{Recursion Theorem}~\citep{kleene52} gives a primitive-recursive $q$ with
	\begin{equation}
		\varphi_{q(c)}(0)\simeq B(c,q(c))(0),
		\label{eq:fixed-point}
	\end{equation}
	Here is the construction used throughout. Let $d(c)$ be the code, computed by primitive-recursive compilation, of the two-argument program
	\[
		(z,x)\longmapsto B(c,s(z,z))(x),
	\]
	and set $q(c):=s(d(c),d(c))$. The correctness of specialization and this fixed program template are provable in $\HA$. Thus $\HA$ proves \eqref{eq:fixed-point}, expressed as equivalence of termination with each output. No soundness assumption on $c$ is needed for this identity. The formation rule admits $q(c)$ in $\Gamma_k\langle C\rangle$ for $c=\code{C}$, the code of $C$ fixed in $\mathcal K$; its derivation contains $c$ and the construction of $q(c)$, not an answer. The presentation $q(c)$ is adversarial to $C$: it is formed from $C$ and contradicts every answer that $C$ can return on it.
\end{definition}

\begin{remark}
	\label{rem:lambda}
	The untyped $\lambda$-calculus provides a related form of self-reference: for every term $F$, a fixed-point combinator satisfies
	\begin{equation}
		\mathsf Y F=F(\mathsf Y F)
		\label{eq:lambda-fixed-point}
	\end{equation}
	\citep{barendregt84}. The equivalence of $\lambda$-definability and recursiveness connects this setting with program computation~\citep{kleene36}. Our construction uses codes, specialization, and execution. It does not require the surrounding formalism to be untyped.

	Strongly normalizing typed calculi, such as the simply typed $\lambda$-calculus, \textsc{Gödel's System T}, and \textsc{Girard's System F}, have no unrestricted fixed-point combinator~\citep{girard89}. This limits internal execution. It does not prevent a typed formalism with suitable arithmetic from representing program codes, finite traces, or the primitive-recursive formation map $q$. Such a representation does not supply a total evaluator for all represented programs. The relevant distinction is between reasoning about coded computations and executing them without restriction.
\end{remark}

\begin{theorem}[Obstruction including the diagnostic branch]
	\label{thm:obstruction}
	Let $k$ be a sound registry, $c_k:=\code{C_k}$, and $q_k:=q(c_k)$. Then
	\begin{equation}
		\begin{aligned}
			C_k(q_k)&\uparrow,\\
			\delta^{\star}_k&:E_{\infty}(q_k),\\
			q_k&\notin W_k,\\
			\neg\mathsf{Cl}_{\Gamma_k}(C_k)&.
		\end{aligned}
		\label{eq:obstruction}
	\end{equation}
\end{theorem}

\begin{proof}
	Suppose $C_k(q_k)$ terminates; Proposition~\ref{prop:starting} supplies the evidence of its answer. If the answer is $(0,w)$, the fixed point and the first row of \eqref{eq:branch-table} give a computation of $\varphi_{q_k}(0)$ with output $1$, contradicting $E_0(q_k)$ by determinism. The answer $(1,w)$ is refuted symmetrically. If the answer is $(\infty,w)$, the third row gives a terminating computation, contradicting $E_{\infty}(q_k)$. Hence $C_k(q_k)\uparrow$. Every terminating run of $\varphi_{q_k}(0)$ contains a return of $C_k(q_k)$, which constructs $\delta^{\star}_k$. A registration of $q_k$ at stage $s$ would stop the search by stage $s$, so $q_k\notin W_k$. Lemma~\ref{lem:local} refutes the closure warrant.
\end{proof}

\subsection{Admission of the diagnosis}

\begin{definition}[Diagnostic admission]
	\label{def:admission}
	For a sound registry $k$ and $\delta:E_{\infty}(e)$, the registry $k\oplus e$ is given by
	\begin{equation}
		\begin{aligned}
			\mathsf{Enum}(k\oplus e,0,x)&:\Longleftrightarrow x=e,\\
			\mathsf{Enum}(k\oplus e,s+1,x)&:\Longleftrightarrow \mathsf{Enum}(k,s,x),
		\end{aligned}
		\label{eq:admission-enum}
	\end{equation}
	with soundness data $\delta$ at stage $0$ and $\delta_k$ shifted by one stage. The diagnosis at $C_k$ is the package $a_k:=(q_k,\delta^{\star}_k)$, and $k^+:=k\oplus q_k$.
\end{definition}

\begin{theorem}[The effective instance is a recurrent barrier]
	\label{thm:effective-recurrent}
	On the class of pairs $(\Gamma_k,C_k)$ with $k$ a sound registry, the formation rule $c\mapsto q(c)$ with diagnoses $a_k$ is a recurrent barrier. In particular, starting from the empty registry, the stages $k_{n+1}:=k_n^+$ are sound and warranted, $C_{k_{n+1}}$ answers $q_{k_n}$, and the indices $q_{k_n}$ are pairwise distinct.
\end{theorem}

\begin{proof}
	(B1) is Theorem~\ref{thm:obstruction}. For (B3), $k^+$ is sound because $\delta^{\star}_k$ supplies the new entry; Proposition~\ref{prop:starting} supplies its warrant, and the search of $C_{k^+}$ at $q_k$ stops at stage $0$. For (B4), computation witnesses keep their meanings, a registration of $k$ at stage $s$ is a registration of $k^+$ at stage $s+1$ with the same soundness data, and $q_k$ is admitted by its registration. The remaining statements are Proposition~\ref{prop:recurrence}.
\end{proof}

\begin{remark}
	Let $\mathsf D:=\{e:\varphi_e(0)\uparrow\}$. Theorem~\ref{thm:obstruction} states that $k\mapsto q_k$ sends every sound registry $W_k\subseteq\mathsf D$ to an element of $\mathsf D\setminus W_k$, so that $\mathsf D$ is a \textsc{Productive Set} with productive function $k\mapsto q_k$~\citep{rogers67}. The effective instance is the familiar unfolding of this function, read as formation, diagnosis, and admission. It is recurrent but not reflexive: a registry asserts nothing about itself, so (B2) and (B3$'$) have no content here. The next section supplies the assertional practice in which they do.
\end{remark}

\section{The Arithmetic Instance}
\label{sec:arithmetic}

\subsection{Proof-producing resolvers}

\begin{definition}[Arithmetical setting]
	\label{def:arithmetical}
	A standard theory of proof is given by a $\Sigma^0_1$ presentation $\tau(x)\equiv\exists w\,\tau_0(x,w)$ of axioms in the language of arithmetic, over intuitionistic or classical logic, with $\tau_0$ primitive recursive, such that $\HA$ proves that every axiom of $\HA$ is an axiom of $\mathrm{T}$. Write $\Prov_{\mathrm{T}}$ for its provability predicate, $\Con(\mathrm{T}):=\neg\Prov_{\mathrm{T}}(\code{\bot})$, and, for an index $e$,
	\begin{equation}
		\Div(e):=\neg\,\exists t\,\exists y\;\mathsf T(e,0,t,y).
		\label{eq:div}
	\end{equation}
	The base theory is $\HA$ with \textsc{Kleene Realizability}~\citep{kleene45,troelstra88}. The reference class $\mathcal R$ consists of the sentences $\exists x\,R(x)$ and $\neg\neg\exists x\,R(x)$ with $R$ primitive recursive; on these, realizability coincides with truth in $\mathbb N$. The assertion relation of $\mathrm{T}$ is $\mathrm{T}$-provability.
\end{definition}

\begin{remark}
	\label{rem:facts}
	For standard theories, $\HA$ proves the \textsc{Hilbert--Bernays--Löb Derivability Conditions}, formalized $\Sigma^0_1$-completeness, the formalized \textsc{Deduction Theorem} $\Prov_{\mathrm{T}}(\code{\neg\varphi})\leftrightarrow\Prov_{\mathrm{T}+\varphi}(\code{\bot})$, hence also $\neg\Prov_{\mathrm{T}}(\code{\neg\varphi})\leftrightarrow\Con(\mathrm{T}+\varphi)$, and the formalized \textsc{Second Incompleteness Theorem} $\Con(\mathrm{T})\to\neg\Prov_{\mathrm{T}}(\code{\Con(\mathrm{T})})$~\citep{goedel31,troelstra88}. These are the only facts about provability used below.
\end{remark}

\begin{definition}[The trap of a theory]
	\label{def:theory-registry}
	Let $k_{\mathrm{T}}$ be the registry with
	\begin{equation}
		\mathsf{Enum}(k_{\mathrm{T}},s,e)
		:\Longleftrightarrow
		\begin{aligned}
			&s\text{ codes a }\mathrm{T}\text{-proof of }\Div(\overline e)\\
			&\text{with a }\tau_0\text{-witness for each axiom used}.
		\end{aligned}
		\label{eq:theory-registry}
	\end{equation}
	The witnesses make $\mathsf{Enum}(k_{\mathrm{T}},s,e)$ decidable, and $\exists s\,\mathsf{Enum}(k_{\mathrm{T}},s,e)$ is $\HA$-provably equivalent to $\Prov_{\mathrm{T}}(\code{\Div(\overline e)})$. Below, a $\mathrm{T}$-proof is always taken with such witnesses.
	Write $C_{\mathrm{T}}:=C_{k_{\mathrm{T}}}$, $q_{\mathrm{T}}:=q(c_{\mathrm{T}})$ with $c_{\mathrm{T}}:=\code{C_{\mathrm{T}}}$, and
	\begin{equation}
		\mathsf{trap}_{\mathrm{T}}:=\Div(\overline{q_{\mathrm{T}}}).
		\label{eq:trap}
	\end{equation}
	Thus $C_{\mathrm{T}}$ searches in parallel for a terminating computation and for a $\mathrm{T}$-proof of divergence, and $\mathsf{trap}_{\mathrm{T}}$ expresses $\mathsf{Bar}(\Gamma_{k_{\mathrm{T}}},C_{\mathrm{T}})$.
\end{definition}

\begin{lemma}[Consistency is the decoder]
	\label{lem:decoder}
	The registry $k_{\mathrm{T}}$ is sound if and only if $\mathrm{T}$ is consistent.
\end{lemma}

\begin{proof}
	A $\mathrm{T}$-proof of $\Div(\overline e)$ together with a terminating computation of $\varphi_e(0)$, which $\mathrm{T}$ proves by $\Sigma^0_1$-completeness, is a $\mathrm{T}$-proof of $\bot$; consistency therefore yields soundness. Conversely, an inconsistent $\mathrm{T}$ registers indices of terminating programs.
\end{proof}

\begin{remark}
	\label{rem:hinge}
	The decoder of a diagnostic certificate is local reflection,
	\begin{equation}
		\Prov_{\mathrm{T}}\bigl(\code{\Div(\overline e)}\bigr)\to\Div(\overline e).
		\label{eq:local-reflection}
	\end{equation}
	Consistency supplies it from outside. By \textsc{Löb's Theorem}~\citep{lob55}, $\mathrm{T}$ proves it only where it is already idle, that is, where $\mathrm{T}$ proves $\Div(\overline e)$.
\end{remark}

\subsection{The trap at its own level}

\begin{theorem}[The trap at its own level]
	\label{thm:trap}
	Let $\mathrm{T}$ be standard. Then
	\begin{equation}
		\begin{aligned}
			\HA&\vdash\mathsf{trap}_{\mathrm{T}}\leftrightarrow\Con(\mathrm{T}),\\
			\HA&\vdash\mathsf{trap}_{\mathrm{T}}\leftrightarrow C_{\mathrm{T}}(\overline{q_{\mathrm{T}}})\uparrow,\\
			\HA&\vdash\neg\Prov_{\mathrm{T}}(\code{\mathsf{trap}_{\mathrm{T}}})\leftrightarrow\mathsf{trap}_{\mathrm{T}}.
		\end{aligned}
		\label{eq:trap-con}
	\end{equation}
	If $\mathrm{T}$ is consistent, then $C_{\mathrm{T}}(q_{\mathrm{T}})\uparrow$, $\mathsf{trap}_{\mathrm{T}}$ is true, and $\mathrm{T}\nvdash\mathsf{trap}_{\mathrm{T}}$.
\end{theorem}

\begin{proof}
	Write $q$ for $q_{\mathrm{T}}$; $\HA$ proves the fixed-point equality \eqref{eq:fixed-point} for the effective construction of $q$, which gives the second equivalence.

	Assume $\Con(\mathrm{T})$ and a terminating run of $\varphi_q(0)$, and decode the answer of $C_{\mathrm{T}}(q)$ contained in it. A computational answer has the wrong parity by \eqref{eq:branch-table}. A diagnostic answer is a $\mathrm{T}$-proof of $\Div(\overline q)$, while $\Sigma^0_1$-completeness turns the run into a $\mathrm{T}$-proof of its negation, contradicting $\Con(\mathrm{T})$. Conversely, $\Prov_{\mathrm{T}}(\code{\bot})$ yields a $\mathrm{T}$-proof of $\Div(\overline q)$, at which the search of $C_{\mathrm{T}}$ stops; then $\varphi_q(0)$ terminates. This gives the first equivalence. The third follows from the first, ex falso, and the formalized \textsc{Second Incompleteness Theorem}.

	If $\mathrm{T}$ is consistent, Lemma~\ref{lem:decoder} and Theorem~\ref{thm:obstruction} give $C_{\mathrm{T}}(q)\uparrow$ and $q\notin W_{k_{\mathrm{T}}}$, that is, $\mathrm{T}\nvdash\mathsf{trap}_{\mathrm{T}}$.
\end{proof}

\begin{remark}
	The third equivalence states that the unprovability of the trap is the trap. The first identifies the failure of the resolver at its own trap with the consistency of the theory that supplies its diagnoses. The trap is true whenever $\mathrm{T}$ is consistent; what fails is the availability of that truth at the level whose resolver it concerns.
\end{remark}

\begin{theorem}[Inner collapse]
	\label{thm:inner}
	If a standard theory $\mathrm{T}$ proves $\mathsf{trap}_{\mathrm{T}}$ or its negation, then $\mathrm{T}$ proves a false sentence of $\mathcal R$.
\end{theorem}

\begin{proof}
	An inconsistent theory proves the false sentence $\exists x\,(x\ne x)$ of $\mathcal R$. If $\mathrm{T}\vdash\mathsf{trap}_{\mathrm{T}}$, then $\mathrm{T}\vdash\Con(\mathrm{T})$, so $\mathrm{T}$ is inconsistent. If $\mathrm{T}\vdash\neg\mathsf{trap}_{\mathrm{T}}$, then $\mathrm{T}\vdash\neg\Con(\mathrm{T})$, which is the sentence $\neg\neg\exists p\,\mathsf{Prf}_{\mathrm{T}}(p,\code{\bot})$ of $\mathcal R$; it is false when $\mathrm{T}$ is consistent, and otherwise $\mathrm{T}$ is inconsistent.
\end{proof}

\begin{remark}
	\label{rem:fork}
	The second case occurs. For consistent $\mathrm{T}$, the theory $\mathrm{T}':=\mathrm{T}+\neg\Con(\mathrm{T})$ is consistent, since otherwise $\mathrm{T}\vdash\neg\neg\Con(\mathrm{T})$, which is $\Con(\mathrm{T})$, contradicting the \textsc{Second Incompleteness Theorem}. It proves $\neg\Con(\mathrm{T}')$, hence $\neg\mathsf{trap}_{\mathrm{T}'}$: it asserts that its own resolver answers its own trap, while $C_{\mathrm{T}'}(q_{\mathrm{T}'})\uparrow$. Both continuations $\mathrm{T}+\mathsf{trap}_{\mathrm{T}}$ and $\mathrm{T}+\neg\mathsf{trap}_{\mathrm{T}}$ can be consistent. This false assertion of resolution supplies no closure warrant.
\end{remark}

\subsection{Admission and the tower}

\begin{definition}[Arithmetical admission]
	\label{def:arith-admission}
	For standard $\mathrm{T}$, let $\mathrm{T}^+$ be presented by
	\begin{equation}
		\tau^+(x):\Longleftrightarrow
		\tau(x)\vee x=\code{\mathsf{trap}_{\mathrm{T}}}.
		\label{eq:tplus}
	\end{equation}
	Put $\mathrm{T}^{(0)}:=\mathrm{T}$ and $\mathrm{T}^{(m+1)}:=\bigl(\mathrm{T}^{(m)}\bigr)^+$. Each $\mathrm{T}^{(m)}$ is standard, and $\HA$ proves that every axiom of $\mathrm{T}$ is an axiom of $\mathrm{T}^{(m)}$.
\end{definition}

\begin{proposition}[Admission is reflection]
	\label{prop:admission}
	Let $\mathrm{T}$ be arithmetically sound. Then $\mathrm{T}^+$ and $\mathrm{T}+\Con(\mathrm{T})$ have the same theorems, $\mathrm{T}^+$ is arithmetically sound, $C_{\mathrm{T}^+}$ answers $q_{\mathrm{T}}$ with a diagnostic certificate and is warranted on $\Gamma_{k_{\mathrm{T}^+}}$, and every registration of $k_{\mathrm{T}}$ remains one of $k_{\mathrm{T}^+}$.
\end{proposition}

\begin{proof}
	The first claim is \eqref{eq:trap-con}, and $\mathsf{trap}_{\mathrm{T}}$ is true. The new axiom has a one-line proof, at which the search of $C_{\mathrm{T}^+}$ stops; every $\mathrm{T}$-proof is a $\mathrm{T}^+$-proof; Lemma~\ref{lem:decoder} and Proposition~\ref{prop:starting} supply soundness and warrant.
\end{proof}

\begin{lemma}[Two identities]
	\label{lem:identities}
	Let $\mathrm{T}$ be standard, and let $S$ be standard with $\HA$ proving that every axiom of $\mathrm{T}$ is an axiom of $S$. Then
	\begin{equation}
		\begin{aligned}
			\HA&\vdash\mathsf{Indep}_{\mathrm{T}}\bigl(\Con(S)\bigr)\leftrightarrow\Con\bigl(\mathrm{T}+\Con(S)\bigr),\\
			\HA&\vdash\Con\bigl(\mathrm{T}+\Con(\mathrm{T}^{(m)})\bigr)\leftrightarrow\Con\bigl(\mathrm{T}^{(m+1)}\bigr).
		\end{aligned}
		\label{eq:identities}
	\end{equation}
\end{lemma}

\begin{proof}
	For the first, $\neg\Prov_{\mathrm{T}}(\code{\neg\Con(S)})$ is equivalent to $\Con(\mathrm{T}+\Con(S))$ by the \textsc{Deduction Theorem}, and $\neg\Prov_{\mathrm{T}}(\code{\Con(S)})$ is equivalent to $\Con(\mathrm{T})$: one direction is ex falso, the other combines $\Con(S)\to\Con(\mathrm{T})$ with the \textsc{Second Incompleteness Theorem}. Since $\Con(\mathrm{T}+\Con(S))$ implies $\Con(\mathrm{T})$, the conjunction reduces to $\Con(\mathrm{T}+\Con(S))$. For the second, $\mathrm{T}^{(m+1)}$ adds to $\tau$ the sentences $\mathsf{trap}_{\mathrm{T}^{(j)}}$ for $j\le m$, each equivalent to $\Con(\mathrm{T}^{(j)})$ by \eqref{eq:trap-con} and implied by $\Con(\mathrm{T}^{(m)})$; each theory therefore proves the finitely many additional axioms of the other.
\end{proof}

\begin{theorem}[The tower]
	\label{thm:tower}
	Let $\mathrm{T}$ be standard, $\mathsf{Indep}^0_{\mathrm{T}}(\varphi):=\varphi$, and $\mathsf{Indep}^{m+1}_{\mathrm{T}}(\varphi):=\mathsf{Indep}_{\mathrm{T}}\bigl(\mathsf{Indep}^m_{\mathrm{T}}(\varphi)\bigr)$. For every $m$,
	\begin{equation}
		\HA\vdash
		\mathsf{Indep}^m_{\mathrm{T}}\bigl(\mathsf{trap}_{\mathrm{T}}\bigr)
		\leftrightarrow
		\Con\bigl(\mathrm{T}^{(m)}\bigr)
		\leftrightarrow
		\mathsf{trap}_{\mathrm{T}^{(m)}}.
		\label{eq:tower}
	\end{equation}
\end{theorem}

\begin{proof}
	The second equivalence is \eqref{eq:trap-con} for $\mathrm{T}^{(m)}$. For the first, induct on $m$. An equivalence proved in $\HA$ is proved in $\mathrm{T}$ and therefore transfers under $\Prov_{\mathrm{T}}$. Hence the induction hypothesis gives
	\begin{equation}
		\begin{aligned}
			\HA\vdash\mathsf{Indep}^{m+1}_{\mathrm{T}}\bigl(\mathsf{trap}_{\mathrm{T}}\bigr)
			&\leftrightarrow\mathsf{Indep}_{\mathrm{T}}\bigl(\Con(\mathrm{T}^{(m)})\bigr)\\
			&\leftrightarrow\Con\bigl(\mathrm{T}+\Con(\mathrm{T}^{(m)})\bigr)\\
			&\leftrightarrow\Con\bigl(\mathrm{T}^{(m+1)}\bigr),
		\end{aligned}
		\label{eq:tower-step}
	\end{equation}
	by Lemma~\ref{lem:identities}.
\end{proof}

\begin{remark}
	Each question of the tower is asked of the fixed theory $\mathrm{T}$, not of $\mathrm{T}^{(m)}$; its answer lies $m$ stages later. For classical $\mathrm{T}$, writing $\Box$ for $\Prov_{\mathrm{T}}$, \eqref{eq:tower} is the familiar identity
	\begin{equation}
		\mathsf{Indep}^m_{\mathrm{T}}(\neg\Box\bot)\leftrightarrow\neg\Box^{m+1}\bot
		\label{eq:provability-logic}
	\end{equation}
	of provability logic. What is added here is that it holds over $\HA$ with the trap in place of $\Con(\mathrm{T})$, and its reading as clause (B3$'$) of Definition~\ref{def:reflexive}.
\end{remark}

\begin{corollary}[Staggered warrant]
	\label{cor:staggered}
	Let $\mathrm{T}$ be arithmetically sound. For every $m$, the sentence $\mathsf{Indep}^m_{\mathrm{T}}(\mathsf{trap}_{\mathrm{T}})$ is true, unprovable in $\mathrm{T}^{(m)}$, and provable in $\mathrm{T}^{(m+1)}$. The trap of $\mathrm{T}$ is settled in $\mathrm{T}^{(1)}$, its independence from $\mathrm{T}$ in $\mathrm{T}^{(2)}$, and the $m$-th iterated question about it in $\mathrm{T}^{(m+1)}$.
\end{corollary}

\begin{proof}
	Every $\mathrm{T}^{(m)}$ is arithmetically sound by Proposition~\ref{prop:admission}. By \eqref{eq:tower}, the sentence is equivalent to $\mathsf{trap}_{\mathrm{T}^{(m)}}$, which is true, unprovable in $\mathrm{T}^{(m)}$, and an axiom of $\mathrm{T}^{(m+1)}$.
\end{proof}

\begin{theorem}[The arithmetic instance is a reflexive barrier]
	\label{thm:arith-reflexive}
	On the pairs $(\Gamma_{k_{\mathrm{T}}},C_{\mathrm{T}})$ with $\mathrm{T}$ standard and arithmetically sound, the formation rule $c\mapsto q(c)$ with diagnoses $\mathrm{T}\mapsto \mathrm{T}^+$ is a reflexive barrier.
\end{theorem}

\begin{proof}
	(B1) is Theorem~\ref{thm:trap}, (B2) is Theorem~\ref{thm:inner}, (B3) and (B4) are Proposition~\ref{prop:admission}, and (B3$'$) is the case $m=1$ of Theorem~\ref{thm:tower}.
\end{proof}

\begin{remark}
	Soundness is used only for (B1), which needs consistency, and for (B3), which needs the soundness of $\mathrm{T}^+$. On the class of the theorem, (B2) holds vacuously, since a sound theory proves neither its trap nor its negation. Its content is Theorem~\ref{thm:inner}, which holds for every standard theory, including the consistent and unsound theories of Remark~\ref{rem:fork}; (B3$'$) likewise holds for every standard theory by Theorem~\ref{thm:tower}.
\end{remark}

\subsection{Recognition and resolution}

\begin{proposition}[Recognition and resolution]
	\label{prop:recognition}
	Let $\mathrm{T}$ be standard.
	\begin{enumerate}[label=\textnormal{(\roman*)}]
		\item The relation \enquote{$\varphi$ is $\HA$-provably equivalent to $\mathsf{trap}_{\mathrm{T}}$}, between sentences $\varphi$ and presentations $\tau$, is $\Sigma^0_1$.
		\item A standard theory proves $\mathsf{trap}_{\mathrm{T}}$ only if it proves $\Con(\mathrm{T})$; a standard extension of $\mathrm{T}$ proving $\mathsf{trap}_{\mathrm{T}}$ proves every theorem of $\mathrm{T}^+$.
	\end{enumerate}
\end{proposition}

\begin{proof}
	(i) The index $q_{\mathrm{T}}$ is computed from $\tau$, and $\HA$-provability is $\Sigma^0_1$. (ii) follows from \eqref{eq:trap-con}.
\end{proof}

\begin{remark}
	Recognition and resolution come apart. If a sentence is $\HA$-provably equivalent to the trap of $\mathrm{T}$, proof search eventually finds that equivalence. Different diagonal routes arrive at the same proposition: a \textsc{Gödel Sentence} for $\mathrm{T}$ is $\HA$-equivalent to $\Con(\mathrm{T})$, hence to $\mathsf{trap}_{\mathrm{T}}$, whereas Rosser sentences are not covered. An extension of $\mathrm{T}$ that proves the trap must prove $\Con(\mathrm{T})$. A reasoner who recognizes the trap of $\mathrm{T}$ and declares it true is operating at the level of $\mathrm{T}^+$; if that reasoner's assertions are those of a standard sound theory $S$, then $\mathsf{trap}_S$ is a true sentence whose identity as its own trap is recognizable, but which the reasoner cannot assert. This is the structure of the argument of \citet{penrose94} together with the standard reply~\citep{benacerraf67}. The framework takes no position on whether any reasoner's assertions are so captured.
\end{remark}

\begin{proposition}[Diophantine presentation]
	\label{prop:diophantine}
	There is a polynomial $P_{\mathrm{T}}(\bar x)$ with integer coefficients, computable from $\tau$, such that
	\begin{equation}
		\HA\vdash
		\neg\exists\bar x\,\bigl(P_{\mathrm{T}}(\bar x)=0\bigr)
		\leftrightarrow
		\mathsf{trap}_{\mathrm{T}}.
		\label{eq:diophantine}
	\end{equation}
\end{proposition}

\begin{proof}
	By the \textsc{Davis--Putnam--Robinson--Matiyasevich Theorem} in uniform form, there is a polynomial $P(e,\bar x)$ such that
	\begin{equation}
		\forall e\,\Bigl(\exists\bar x\,\bigl(P(e,\bar x)=0\bigr)
		\leftrightarrow
		\exists t\,\exists y\,\mathsf T(e,0,t,y)\Bigr)
		\label{eq:mrdp}
	\end{equation}
	is provable in $\mathsf{I}\Delta_0+\mathsf{exp}$~\citep{gaifman82}. Each direction is equivalent to a $\Pi^0_2$ sentence, and $\mathsf{PA}$ is $\Pi^0_2$-conservative over $\HA$~\citep{friedman78}. Put $P_{\mathrm{T}}(\bar x):=P(\overline{q_{\mathrm{T}}},\bar x)$.
\end{proof}

\begin{remark}
	The trap is an ordinary arithmetical problem: a polynomial without a root. It differs from a polynomial whose solvability is merely unknown in that its answer is known one level up and unavailable at its own level. For any transformed equation whose equivalence to the trap is $\HA$-provable, proof search eventually finds that equivalence.
\end{remark}

\section{No Effective Closure}
\label{sec:closure}

\subsection{Effective policies}

\begin{definition}[Effective admission policy]
	\label{def:policy}
	An effective admission policy $\pi$ is a $\Sigma^0_1$ set $I_\pi$ of indices of standard presentations containing an explicitly given index. Its union $\mathrm{T}_\pi$ is presented by
	\begin{equation}
		\tau_\pi(x):\Longleftrightarrow
		\exists i\,\bigl(i\in I_\pi\wedge\tau_i(x)\bigr),
		\label{eq:policy-union}
	\end{equation}
	and is standard. Finite sequences of admissions, the progression $\mathrm{T}^{(m)}$, and progressions indexed by any effectively enumerable system of stages are effective admission policies.
\end{definition}

\begin{corollary}[No effective closure]
	\label{cor:no-closure}
	Let $\pi$ be an effective admission policy with consistent union. Then $\mathsf{trap}_{\mathrm{T}_\pi}$ is true and unprovable in $\mathrm{T}_\pi$, and no stage resolver $C_{\mathrm{T}_i}$, with $\mathrm{T}_i$ presented by $\tau_i$ for $i\in I_\pi$, answers $q_{\mathrm{T}_\pi}$. The policy $\pi^+$ obtained by adding an index of $\mathrm{T}_\pi^+$ is effective; if $\mathrm{T}_\pi$ is arithmetically sound, so is the union of $\pi^+$, and the corollary applies again.
\end{corollary}

\begin{proof}
	Theorem~\ref{thm:trap} applies to $\mathrm{T}_\pi$. A stage resolver cannot return a computational answer at the divergent index $q_{\mathrm{T}_\pi}$, and a diagnostic answer would be a $\tau_i$-proof, hence a $\mathrm{T}_\pi$-proof, of $\mathsf{trap}_{\mathrm{T}_\pi}$. Adding one index preserves $\Sigma^0_1$, and Proposition~\ref{prop:admission} preserves soundness.
\end{proof}

\begin{remark}
	Whatever an effectively specified practice of admission eventually registers, the formation rule applied to the resulting resolver produces a presentation whose trap is true and lies above all its stages, and admitting it continues the policy instead of ending it. The corollary makes no use of how the stages are indexed. Progressions of theories along ordinal notations~\citep{turing39,feferman62} organize the same regress through notations; their completeness results concern paths that are not effectively selected, in agreement with Corollary~\ref{cor:no-closure}.
\end{remark}

\subsection{Where closure lives}

\begin{proposition}[Closure without a code]
	\label{prop:no-code}
	Reason classically, and let $\mathsf D:=\{e:\varphi_e(0)\uparrow\}$. The answer function
	\begin{equation}
		A_{\mathsf D}(e):=
		\begin{cases}
			(\infty,\ast)&\text{if }e\in\mathsf D,\\
			(y\bmod2,t)&\text{if }\mathsf T(e,0,t,y),
		\end{cases}
		\label{eq:oracle-answer}
	\end{equation}
	is total and sound. It is not computable, and for every effective policy $\pi$ with consistent union, $W_{k_{\mathrm{T}_\pi}}$ is a proper subset of $\mathsf D$.
\end{proposition}

\begin{proof}
	Totality and soundness follow from the classical case distinction. Computability would decide the halting problem at input $0$. The last claim combines Lemma~\ref{lem:decoder} with Corollary~\ref{cor:no-closure}.
\end{proof}

\begin{remark}
	The function $A_{\mathsf D}$ gives total, sound answers, but has no program index in the fixed numbering. It therefore lies outside the effective resolution setting. By contrast, the consistent theory in Remark~\ref{rem:fork} falsely asserts that its resolver answers its trap, although that resolver diverges there. This assertion is not closure in the sense of \eqref{eq:closure}. If operations computable relative to $\mathsf D$ are admitted, the relativized \textsc{Recursion Theorem} produces a new trap at the level of the \textsc{Turing Jump}.
\end{remark}

\begin{remark}
	The trap depends on the presentation $\tau$, not only on the theorems of $\mathrm{T}$; different presentations of the same theory give different resolvers and different traps. The identification \eqref{eq:trap-con} relies on the derivability conditions, which hold for standard presentations; for nonstandard presentations the \textsc{Second Incompleteness Theorem} can fail in its usual form~\citep{feferman60}. Admission is an operation on presented theories, and the regress branches already at the choice of presentation.
\end{remark}

\section{The Feasible Regime}
\label{sec:feasible}

\subsection{What changes at feasible scale}

\begin{remark}
	The unbounded construction imposes no resource bound on the fixed point or its execution. At feasible scale, the cost of encoding computations matters, while bounded consistency can have short proofs. We now work with formulas, proofs, and theories. Polynomial-time requirements concern proof construction and verification in the sense of \citet{cookreckhow79}.
\end{remark}

\begin{definition}[Proof systems and the diagnostic act]
	\label{def:proof-system}
	A propositional proof system is a polynomial-time computable function $P$ from strings onto the set $\mathsf{TAUT}$ of propositional tautologies; a string $w$ with $P(w)=\psi$ is a $P$-proof of $\psi$. The system $P$ is polynomially bounded if every $\psi\in\mathsf{TAUT}$ has a $P$-proof of size polynomial in $|\psi|$. It is optimal if for every proof system $Q$ there is a polynomial $p$ such that every tautology with a $Q$-proof of size $s$ has a $P$-proof of size at most $p(s)$.

	In the feasible practice $\Gamma^{\mathsf f}_P$, presentations are propositional formulas $\varphi$. The accepted answers are a satisfying assignment of $\varphi$, and a $P$-proof of $\neg\varphi$ of size bounded by a fixed polynomial in $|\varphi|$. The second answer is the diagnostic act: a short certificate that $\varphi$ has no solution.
\end{definition}

\begin{proposition}[Feasible obstruction of the diagnostic act]
	\label{prop:cook-reckhow}
	The following are equivalent:
	\begin{enumerate}[label=\textnormal{(\roman*)}]
		\item for every proof system $P$ and every polynomial bound, some unsatisfiable formula has no accepted diagnostic answer in $\Gamma^{\mathsf f}_P$;
		\item no polynomially bounded proof system exists;
		\item $\mathcal{NP}\ne\mathrm{co}\mathcal{NP}$.
	\end{enumerate}
\end{proposition}

\begin{proof}
	(i) and (ii) restate each other. The equivalence of (ii) and (iii) is the theorem of \citet{cookreckhow79}.
\end{proof}

\begin{remark}
	Proposition~\ref{prop:cook-reckhow} is the feasible form of clause (B1), with the formation rule removed: it asserts the existence of unanswerable presentations but does not form them from the resolver. Its truth is open. The unbounded construction supplied (B1) by a fixed point at no cost; at feasible scale, the propositional form of (B1) is itself the central open problem, and the reflexive question is better asked of clauses (B2) and (B3$'$). A first-order bounded form of (B1) holds unconditionally by Proposition~\ref{prop:bounded-trap}, but only relative to the length of the name of its window.
\end{remark}

\subsection{The split}

\begin{definition}[Finitistic consistency]
	\label{def:fincon}
	For a theory $S$ and a numeral $\overline n$, let $\Con_S(\overline n)$ state that no $S$-proof of $\bot$ has length at most $n$. For a propositional proof system $P$, let $\mathsf{Ref}_P$ be the sequence of tautologies expressing that $P$-proofs of a given size prove only tautologies, and let $P+\mathsf{Ref}_P$ be the proof system whose proofs are $P$-proofs and derivations in a fixed \textsc{Frege System} from substitution instances of $\mathsf{Ref}_P$.
\end{definition}

\begin{proposition}[Self-consistency is feasibly cheap]
	\label{prop:cheap}
	The following results are known.
	\begin{enumerate}[label=\textnormal{(\roman*)}]
		\item Let $\mathrm{T}$ be a finitely axiomatized sequential theory. There is a polynomial-time computable sequence of $\mathrm{T}$-proofs of $\Con_{\mathrm{T}}(\overline n)$, and $\mathsf S^1_2$ proves that it is such a sequence~\citep{pudlak86}, \citep[Thm.~4.8]{pudlak20}.
		\item A Frege system has polynomial-size proofs of its own consistency statements~\citep{buss91}. The \textsc{Extended Frege System} $\mathsf{EF}$ has polynomial-size proofs of its own reflection principle, since $\mathsf{PV}$ proves it and $\mathsf{PV}$-proofs translate into polynomial-size $\mathsf{EF}$-proofs~\citep{cook75,krajicek95}; and for every system $P$ extending \textsc{Circuit Frege}, the system $P+\mathsf{Ref}_P$ provably has polynomial-size proofs of its own reflection principle~\citep[Cor.~4.4]{pudlak20}.
	\end{enumerate}
\end{proposition}

\begin{remark}
	Both results are cited, not reproved. Together they show that a sufficiently strong theory or proof system certifies its own bounded consistency at polynomial cost. In the unbounded instance, the trap of $\mathrm{T}$ is $\HA$-provably $\Con(\mathrm{T})$. At feasible scale, the bounded version of $\Con(\mathrm{T})$ is cheap, and the feasible obstruction, if there is one, must lie elsewhere. Apart from the clocked \textsc{Gödel Sentence} of Proposition~\ref{prop:bounded-trap}, bounded forms are taken here of consistency statements, not of the divergence statement $\Div(\overline{q_{\mathrm{T}}})$; the equivalence \eqref{eq:trap-con} is not claimed to hold with polynomial bounds.
\end{remark}

\begin{proposition}[The next level]
	\label{prop:next-level}
	Let $\mathrm{T}$ be a computably axiomatized theory containing $\mathsf S^1_2$ such that $\mathrm{T}+\Con(\mathrm{T})$ is consistent. For no polynomial-time computable sequence $(E_n)_n$ does $\mathrm{T}+\Con(\mathrm{T})$ prove that, for all $n$, $E_n$ is a $\mathrm{T}$-proof of $\Con_{\mathrm{T}+\Con(\mathrm{T})}(\overline n)$~\citep[Prop.~4.9]{pudlak20}.
\end{proposition}

\begin{remark}
	\label{rem:pudlak-conjecture}
	\citet{pudlak20} conjectures that, for $\mathrm{T}$ containing $\mathsf S^1_2$ with $\mathrm{T}+\Con(\mathrm{T})$ consistent, the sentences $\Con_{\mathrm{T}+\Con(\mathrm{T})}(\overline n)$ have no $\mathrm{T}$-proofs of length polynomial in $n$, and, in a weaker form, that such proofs cannot be constructed in polynomial time; see also \citet{pudlak17}. For standard $\mathrm{T}$, Lemma~\ref{lem:identities} and Theorem~\ref{thm:tower} give
	\begin{equation}
		\HA\vdash
		\Con\bigl(\mathrm{T}+\Con(\mathrm{T})\bigr)
		\leftrightarrow
		\mathsf{Indep}_{\mathrm{T}}\bigl(\mathsf{trap}_{\mathrm{T}}\bigr)
		\leftrightarrow
		\mathsf{trap}_{\mathrm{T}^+}.
		\label{eq:pudlak-tower}
	\end{equation}
	The first identity of Lemma~\ref{lem:identities} is stated for standard theories, but its proof uses only the derivability conditions and the formalized \textsc{Second Incompleteness Theorem}. For a finitely axiomatized sequential theory $S$ containing $\mathsf S^1_2$, which supplies both provably in $\mathsf{PA}$, the same proof gives
	\begin{equation}
		\mathsf{PA}\vdash
		\Con\bigl(S+\Con(S)\bigr)
		\leftrightarrow
		\mathsf{Indep}_{S}\bigl(\Con(S)\bigr).
		\label{eq:pudlak-indep}
	\end{equation}
	The comparison is made at the level of consistency statements, where \eqref{eq:trap-con} locates the trap. The identity motivates reading bounded consistency of $S+\Con(S)$ as a bounded warrant of the independence of $\Con(S)$. It does not establish a bounded version of clause (B3$'$), which would concern independence of the proposed feasible obstruction itself. What is cheap is bounded self-consistency; what is conjecturally expensive is bounded consistency of the next theory. This is the sense in which the proposed analogue moves one level up.
\end{remark}

\begin{remark}
	\label{rem:split}
	The displacement has a simple source. Two statements of absence must be distinguished:
	\begin{equation}
		\begin{aligned}
			&\text{no witness of a contradiction exists}
			&&\text{(consistency)},\\
			&\text{no short witness of a property exists}
			&&\text{(hardness)}.
		\end{aligned}
		\label{eq:split}
	\end{equation}
	In the unbounded regime only the first is needed: Theorem~\ref{thm:trap} reduces the warrant that $\varphi_{q_{\mathrm{T}}}(0)$ has no terminating computation to $\Con(\mathrm{T})$. At feasible scale the two come apart. Bounded consistency is a statement about one finite search space whose exhaustion a strong theory certifies efficiently, by Proposition~\ref{prop:cheap}. Hardness is a statement about all short witnesses for all instances of every length, and its warrant inside a feasible theory is the subject of the conditional results below. The motivation for a feasible analogue lies on the side of hardness.
\end{remark}

\subsection{Hardness against the warrant of hardness}

\begin{proposition}[Conditional barriers to certifying hardness]
	\label{prop:conditional}
	The following results are known.
	\begin{enumerate}[label=\textnormal{(\roman*)}]
		\item If strong pseudorandom generators exist, then the theory $\mathsf S^2_2(\alpha)$ of \textsc{Bounded Arithmetic} cannot refute that a circuit of size $n^{\log^\ast n}$ decides satisfiability~\citep{razborov95}.
		\item If strong pseudorandom functions exist, then no natural property is useful against polynomial-size circuits~\citep{rr97}.
	\end{enumerate}
\end{proposition}

\begin{remark}
	These results motivate a comparison with inner collapse under a hypothesis $H$. The practice is a feasible theory, or the class of natural proof methods; the proposed conclusion is a circuit lower bound; and a warrant of that lower bound inside the practice would refute $H$. The hypothesis $H$ is itself a hardness statement. In this sense hardness obstructs the warrant of hardness. The comparison is limited: the results concern certification of a lower bound, whereas (B2) concerns assertions of either the barrier or its negation in a specified assertional practice and reference class. They do not establish (B2), nor do they supply the formation and admission rules required by Definition~\ref{def:reflexive}.
\end{remark}

\subsection{The cost of a fixed point}

\begin{definition}[Bounded feedback]
	\label{def:bounded-feedback}
	Fix a deterministic \textsc{Turing Machine} model. Let $V(e,t)$ be a propositional formula, computed from $e$ and $t$, satisfiable exactly when $\varphi_e(0)$ returns $1$ within $t$ steps, as in the \textsc{Cook--Levin Theorem}~\citep{cook71}. Let $R$ be a total correct satisfiability decider. Effective specialization and the \textsc{Recursion Theorem} give an index $b_t$ with
	\begin{equation}
		\begin{aligned}
			F_t&:=V(b_t,t),\\
			\varphi_{b_t}(0)&\simeq 1-R(F_t),\\
			\mathsf{SAT}(F_t)
			&\longleftrightarrow
			\bigl(\varphi_{b_t}(0)\text{ returns }1\text{ within }t\bigr).
		\end{aligned}
		\label{eq:bounded-feedback}
	\end{equation}
	The time measure includes every transition of the compiled program $b_t$.
\end{definition}

\begin{proposition}[Escape from the verification window]
	\label{prop:bounded}
	For every $t$, the formula $F_t$ is unsatisfiable, and $\varphi_{b_t}(0)$ returns $1$ after more than $t$ steps.
\end{proposition}

\begin{proof}
	The program terminates since $R$ is total. If $R(F_t)=1$, then $F_t$ is satisfiable, so the program returns $1$ within $t$ steps, while the second line makes it return $0$. Hence $R(F_t)=0$, $F_t$ is unsatisfiable, the program returns $1$, and the last line excludes a return within $t$ steps.
\end{proof}

\begin{remark}
	For an encoding with at least one variable per simulated step, $|F_t|\ge t$. An implementation that writes $F_t$ before calling $R$ therefore already spends its window on the formula. Proposition~\ref{prop:bounded} concerns this feedback construction, not every way to form a counterexample to $R$. Such diagonal arguments relativize, so the result of \citet{bks75} excludes using them alone to settle $\mathcal P$ versus $\mathcal{NP}$.
\end{remark}

\begin{proposition}[Bounded trap]
	\label{prop:bounded-trap}
	Let $\mathrm{T}$ be standard, with binary numerals and the witnessed proof predicate of Definition~\ref{def:theory-registry}. Proof length counts the whole certificate, including its axiom witnesses. Use the primitive-recursive proof checker with its $\HA$-provably correct representation. The parameterized diagonal lemma gives a formula $\psi(x)$ satisfying
	\begin{equation}
		\HA\vdash\psi(x)\leftrightarrow
		\neg\exists p\,\bigl(|p|\le x\wedge\mathsf{Prf}_{\mathrm{T}}(p,\code{\psi(\dot x)})\bigr),
		\label{eq:bounded-trap}
	\end{equation}
	so that $|\psi(\overline n)|=O(\log n)$. For every $n$:
	\begin{enumerate}[label=\textnormal{(\roman*)}]
		\item if $\mathrm{T}$ is consistent, then $\psi(\overline n)$ is true and $\mathrm{T}$-provable, and has no $\mathrm{T}$-proof of length at most $n$;
		\item a $\mathrm{T}$-proof of $\psi(\overline n)$ of length at most $n$, or any $\mathrm{T}$-proof of $\neg\psi(\overline n)$, makes $\mathrm{T}$ inconsistent;
		\item $\mathrm{T}+\Con(\mathrm{T})$ proves $\psi(\overline n)$ with proofs of length $O(\log n)$.
	\end{enumerate}
\end{proposition}

\begin{proof}
	For a code $z$ of a one-variable formula, let $H(x,z)$ assert that some proof of length at most $x$ proves its instance at $\overline x$. The search checks finitely many strings with the primitive-recursive proof checker. Its characteristic function is primitive recursive, and $\HA$ proves its correctness and totality. In particular, $H$ has a $\Sigma^0_1$ representation that $\HA$ proves decidable. Applying the parameterized diagonal lemma to $\neg H(x,z)$ gives \eqref{eq:bounded-trap}. This also makes $\psi(x)$ $\HA$-provably decidable. No claim that $\psi$ is syntactically bounded is needed.

	For each standard $n$, $\HA$ can verify the complete finite search at $n$. It therefore proves $\psi(\overline n)$ when the search finds no proof, and proves $\neg\psi(\overline n)$ when it finds one.

	(ii) If a $\mathrm{T}$-proof of $\psi(\overline n)$ has length at most $n$, the search finds a proof. The preceding verification gives an $\HA$-proof of its negation, so $\mathrm{T}$ is inconsistent. Now suppose $\mathrm{T}$ proves $\neg\psi(\overline n)$. If the search finds a proof, the first case applies. Otherwise $\HA$ proves $\psi(\overline n)$, again making $\mathrm{T}$ inconsistent.

	(i) Consistency excludes a proof of length at most $n$ by (ii). Hence the search finds none, and $\psi(\overline n)$ is true and $\HA$-provable.

	(iii) Work in $\mathrm{T}+\Con(\mathrm{T})$. If $\neg\psi(x)$, decidability and \eqref{eq:bounded-trap} give $H(x,\code{\psi})$, hence a $\mathrm{T}$-proof of $\psi(\dot x)$. Formalized $\Sigma^0_1$-completeness gives a $\mathrm{T}$-proof of the corresponding instance of $H$. The fixed-point equivalence converts this into a $\mathrm{T}$-proof of $\neg\psi(\dot x)$. This contradicts $\Con(\mathrm{T})$. Decidability now yields $\psi(x)$. Thus one fixed proof establishes $\forall x\,\psi(x)$, and instantiation at the binary numeral $\overline n$ gives proofs of length $O(\log n)$.
\end{proof}

\begin{remark}
	Proposition~\ref{prop:bounded-trap} gives a bounded counterpart of the third equivalence in \eqref{eq:trap-con}: bounded unprovability is the sentence itself. It supplies bounded forms of (B1), (B2), and the answering part of (B3), without a complexity-theoretic hypothesis. The lower bound is $n$, while the admitted theory has proofs of length $O(\log n)$. This is a classical speed-up phenomenon~\citep{pudlak98}, but it does not establish a superpolynomial lower bound in $n$. Conjecture~\ref{conj:feasible} asks for that stronger lower bound for a different family of sentences.
\end{remark}

\begin{remark}
	The statement of $\mathcal P\ne\mathcal{NP}$ needs no fixed point. With machines clocked by a polynomial and satisfiability decided by exhaustive search, it is a $\Pi^0_2$ sentence with primitive recursive matrix, and it is a type of a strongly normalizing calculus:
	\begin{equation}
		\prod_{M}\prod_{k}\sum_{\varphi}
		\bigl(M\text{ clocked at }|\varphi|^k+k\text{ errs on }\varphi\bigr).
		\label{eq:pnp-type}
	\end{equation}
	An inhabitant is a total formation rule: from the code of a proposed resolver it forms a presentation the resolver does not answer. A proof in $\mathsf{PA}$ supplies one, since $\mathsf{PA}$ is $\Pi^0_2$-conservative over $\HA$~\citep{friedman78} and \textsc{Gödel's Dialectica Interpretation} extracts a term of \textsc{Gödel's System T}~\citep{goedel58}. For $\mathcal P$ versus $\mathcal{EXP}$ the corresponding type is inhabited through the \textsc{Time Hierarchy Theorem}, whose diagonalizer can afford to simulate the resolver. Proposition~\ref{prop:bounded} shows that the specific satisfiability feedback program exceeds its verification window. Equation~\eqref{eq:pnp-type} places no time bound on an inhabitant. Under $\mathcal P\ne\mathcal{NP}$, exhaustive search for a disagreement terminates, even if it fully simulates each clocked run and decides satisfiability by brute force.
\end{remark}

\subsection{The hierarchy test}

\begin{criterion}[Hierarchy test]
	\label{crit:hierarchy}
	Consider an argument whose hypotheses mention two classes only through their arithmetical representability, and which concludes that their separation cannot be established in a given theory. Instantiate the classes with $\mathcal P$ and $\mathcal{EXP}$. The \textsc{Time Hierarchy Theorem}~\citep{hartmanis65} separates them by an explicit diagonal argument formalizable in $\HA$. If this separation meets the argument's notion of separation, the argument is refuted by this instance. A barrier specific to $\mathcal P$ versus $\mathcal{NP}$ must use a hypothesis that fails for $\mathcal P$ versus $\mathcal{EXP}$.
\end{criterion}

\begin{remark}
	Sections~\ref{sec:reflexive}--\ref{sec:closure} pass the test by making no claim about any separation: their objects are resolvers, registries, and admission policies, and their conclusions hold verbatim for tasks concerning $\mathcal{EXP}$. The statements of this section pass it by their content: polynomially bounded proofs of tautologies, bounded consistency, and circuit lower bounds for satisfiability are specific to $\mathcal{NP}$, and none of them is settled by a hierarchy argument.
\end{remark}

\subsection{A motivated feasible analogue}

\begin{proposition}[The propositional candidate fails]
	\label{prop:prop-candidate}
	For $P=\mathsf{EF}$, every sequence of tautologies with $(P+\mathsf{Ref}_P)$-proofs of size polynomial in $n$ has $P$-proofs of size polynomial in $n$.
\end{proposition}

\begin{proof}
	By Proposition~\ref{prop:cheap}, $\mathsf{EF}$ has polynomial-size proofs of $\mathsf{Ref}_{\mathsf{EF}}$. A substitution instance of an $\mathsf{EF}$-proof is an $\mathsf{EF}$-proof of the substituted formula, with polynomial increase in size, and $\mathsf{EF}$ simulates every \textsc{Frege System}. Replace each substitution instance of $\mathsf{Ref}_{\mathsf{EF}}$ in a $(P+\mathsf{Ref}_P)$-proof by the corresponding instance of its $\mathsf{EF}$-proof, and the \textsc{Frege} derivation by its simulation.
\end{proof}

\begin{remark}
	A propositional form of Definition~\ref{def:reflexive} would let the formation rule act on the code of a proof system $P$ and take admission to be $P\mapsto P+\mathsf{Ref}_P$. Proposition~\ref{prop:prop-candidate} shows that no formation rule can satisfy (B1) and (B3) in this form at $\mathsf{EF}$: for strong systems the admission of reflection adds nothing at polynomial scale. Remark~\ref{rem:pudlak-conjecture} indicates the level at which an analogue should be sought. Admission is the passage $S\mapsto S+\Con(S)$, and the formed instance is the bounded consistency of the admitted theory.
\end{remark}

\begin{conjecture}[Bounded warrant of independence; Pudlák]
	\label{conj:feasible}
	For every finitely axiomatized sequential theory $S$ containing $\mathsf S^1_2$ such that $S+\Con(S)$ is consistent, the sentences $\Con_{S+\Con(S)}(\overline n)$ have no $S$-proofs of length polynomial in $n$.
\end{conjecture}

\begin{remark}
	This is the conjecture of \citet{pudlak20} recorded in Remark~\ref{rem:pudlak-conjecture}, restricted to finitely axiomatized theories. It motivates a feasible analogue of Definition~\ref{def:reflexive}; the following correspondences and limits explain that reading.

	The proposed formation sends $S$ to the family $\bigl(\Con_{S+\Con(S)}(\overline n)\bigr)_n$, and admission sends $S$ to $S+\Con(S)$. The conjecture supplies the proposed analogue of (B1): this family has no $S$-proofs of length polynomial in $n$. The answering part of (B3) is a theorem of \citet{pudlak86}, as stated in Proposition~\ref{prop:cheap}: the admitted theory is finitely axiomatized and sequential, and has proofs of these sentences of length polynomial in $n$. This does not verify the whole of (B3). Membership in the conjecture's class is not known to be preserved by admission, since consistency of $S+\Con(S)$ does not imply consistency of its own consistency extension.

	Equation~\eqref{eq:pudlak-indep} motivates the comparison with (B3$'$): the unbounded consistency statement whose finite instances are considered is equivalent to independence of $\Con(S)$ from $S$. As noted in Remark~\ref{rem:pudlak-conjecture}, this does not yield a bounded version of (B3$'$), and no feasible version of (B2) is given. Definition~\ref{def:reflexive} therefore applies to the unbounded theorems; here it guides an argument for where a feasible analogue should be sought.

	The conjecture specifies the theory whose bounded consistency is at issue, namely $S+\Con(S)$. This distinguishes it from the existence, for each $S$, of some finitely axiomatized consistent theory whose finitistic consistency has no short $S$-proofs, which \citet{krajicekpudlak89} showed equivalent to the nonexistence of an optimal proof system~\citep{pudlak17}.
\end{remark}

\begin{proposition}[Consequences]
	\label{prop:conj-consequences}
	Conjecture~\ref{conj:feasible} implies that no optimal propositional proof system exists, and hence that $\mathcal{NP}\ne\mathrm{co}\mathcal{NP}$ and $\mathcal P\ne\mathcal{NP}$.
\end{proposition}

\begin{proof}
	The argument is that of \citet{krajicekpudlak89}. Suppose $P$ is optimal. Let $S$ be $\mathsf S^1_2$ together with the reflection principle $\forall w\,\mathsf{Taut}(P(w))$ of $P$, a single true sentence; $S$ is finitely axiomatized, sequential, and true, so $\mathrm{U}:=S+\Con(S)$ is consistent. Let $\sigma_n$ be the propositional formulas, constructible in polynomial time, expressing that no string of length at most $n$ is a $\mathrm{U}$-proof of $\bot$. They are tautologies, and the proof system that accepts $1^n$ as a proof of $\sigma_n$, and otherwise behaves as $P$, gives them proofs of size $n$. By optimality, they have $P$-proofs $\pi_n$ of size polynomial in $n$. In $S$, the sentence $\Con_{\mathrm{U}}(\overline n)$ follows from a verification of $P(\pi_n)=\sigma_n$, the axiom of reflection, and the implication from $\mathsf{Taut}(\sigma_n)$ to $\Con_{\mathrm{U}}(\overline n)$, each of length polynomial in $n$. This contradicts the conjecture at $S$. If $\mathcal{NP}=\mathrm{co}\mathcal{NP}$, a polynomially bounded proof system exists by Proposition~\ref{prop:cook-reckhow}, and it is optimal. Finally, $\mathcal P=\mathcal{NP}$ implies $\mathcal{NP}=\mathrm{co}\mathcal{NP}$.
\end{proof}

\section{Conclusion}
\label{sec:conclusion}

A resolver's code can enter the formation of a problem that the resolver cannot answer. For programs, productiveness supplies this obstruction and a sound admission step. For standard arithmetical theories, the trap is $\HA$-provably equivalent to consistency. Its independence from the original theory is equivalent to the trap of the consistency extension. Iteration gives the tower of Theorem~\ref{thm:tower}.

These results distinguish recognition, proof, and truth. Proof search can recognize an equivalence to a trap, but the trap's truth requires an external consistency warrant. A theory that proves its trap is inconsistent. A consistent theory can refute its trap only by asserting a false sentence. Such an assertion supplies no resolution warrant, and no effective admission policy with consistent union achieves closure.

At feasible scale, the clocked \textsc{Gödel Sentence} gives a proof-length gap relative to its short binary description. It gives no superpolynomial lower bound in the window size. Bounded self-consistency and propositional self-reflection can have short proofs, which blocks the direct analogues considered here. The independence identity instead motivates studying bounded consistency of the successor theory. Pudlák's conjecture predicts that these sentences lack polynomial-length proofs in the original theory. Its finitely axiomatized case implies $\mathcal{NP}\ne\mathrm{co}\mathcal{NP}$ by Proposition~\ref{prop:conj-consequences}, but remains a conjectural analogue of the unbounded barrier.

Most importantly, a question concerns the concealment of formation itself. Could a constructor efficiently produce a presentation and retain a short certificate identifying it as a trap, while recognition without that certificate is hard? With the certificate and the relevant consistency warrant, the constructor could diagnose the obstruction, allowing yet another layer of recursion. A reasoner would face an apparently ordinary but adversarial arithmetic problem. Eventual recognition by proof search does not ensure feasible recognition. In this sense, problem formation could acquire a cryptographic character: access to its formation certificate could separate feasible diagnosis from costly recognition. The present results point to this question.

\clearpage

{\scriptsize
	\bibliographystyle{plainnat}
	\setlength{\bibsep}{0em}
	\bibliography{refs}
}

\clearpage

\vspace*{\fill}

\subsection*{Note on the Fourth Version}

This fourth version replaces the third. The barrier theorem of the third version reached its contradiction only under an added local reflection principle at the diagonal sentence, which is precisely what is unavailable; its corollary that uniform class separation cannot be established in $\HA$ or its extensions is withdrawn. The diagonal classifier is replaced by a fixed formation rule acting on resolvers, and the obstruction is now proved without assuming reflection inside the theory: for programs as the unfolding of a productive set, and for arithmetical theories as the identity of the trap with $\Con(\mathrm{T})$, with inner collapse, the tower of iterated independence, and the absence of effective closure. The feasible regime is treated separately: its unconditional part is the clocked \textsc{Gödel Sentence}, and its analogue of the barrier is conditional on a conjecture of Pudlák. Criterion~\ref{crit:hierarchy} shows why the argument of the third version could not have supported one. The title has changed accordingly.

\subsection*{Final Remarks}

The author welcomes criticism, proposed extensions, scholarly correspondence, and constructive dialogue. No conflicts of interest are declared. This research received no funding.

\begin{center}
	\scriptsize{
	\vspace{1em}
	Milan Rosko\\
	\vspace{1em}
	ORCID: \href{https://orcid.org/0009-0003-1363-7158}{\textsf{0009-0003-1363-7158}}\\[1ex]
	Email: \href{mailto:hi*atsymbol*milanrosko.com}{\textsf{hi*atsymbol*milanrosko.com}}\\[1ex]

	\vspace{1em}
	Licensed under \ccby \\ \href{http://creativecommons.org/licenses/by/4.0/}{\scriptsize\textsf{creativecommons.org/licenses/by/4.0}}
	}
\end{center}
\vspace*{\fill}

\end{document}